\documentclass[12pt,a4paper]{amsart}
\usepackage{amssymb,enumerate,mathrsfs,stmaryrd}

\newtheorem{theorem}{Theorem}
\newtheorem*{lemma}{Lemma}
\newtheorem*{corollary}{Corollary}
\newtheorem*{conjecture}{Conjecture}

\newcommand{\vertbar}{\>|\>}
\newcommand{\set}[2]{\ensuremath{\{ #1 \vertbar #2 \}}}

\DeclareMathOperator{\ad}{\mathsf{ad}}
\DeclareMathOperator{\C}{C}
\DeclareMathOperator{\Der}{Der}
\DeclareMathOperator{\diag}{diag}
\DeclareMathOperator{\gl}{\mathsf{gl}}
\DeclareMathOperator{\Rad}{Rad}
\DeclareMathOperator{\rk}{rk}
\DeclareMathOperator{\Sl}{\mathsf{sl}}
\DeclareMathOperator{\sym}{S}
\DeclareMathOperator{\Z}{Z}

\begin{document}

\title{Lie algebras with given properties of subalgebras and elements}
\author{Pasha Zusmanovich}
\address{
Department of Mathematics, Tallinn University of Technology, Ehitajate tee 5, 
Tallinn 19086, Estonia
}
\email{pasha.zusmanovich@ttu.ee}
\date{last revised October 2, 2012}
\thanks{\textsf{arXiv:1105.4284}; Proceedings of the Conference 
``Algebra - Geometry - Mathematical Physics'' (Mulhouse, 2011), 
Springer Proceedings in Mathematics and Statistics, to appear.
}

\maketitle

We study the following classes of Lie algebras: 
anisotropic (i.e., Lie algebras for which each adjoint operator $\ad x$ is 
semisimple),
regular (i.e., Lie algebras in which each nonzero element is regular), 
minimal nonabelian
(i.e., nonabelian Lie algebras all whose proper subalgebras are abelian),
and algebras of depth $2$ (i.e., Lie algebras all whose proper subalgebras are abelian
or minimal nonabelian).

All algebras, Lie and associative, are assumed to be finite-dimensional and 
defined over a fixed field of characteristic zero (though some of the results, 
in a weaker form or under additional restrictions, will hold also in positive 
characteristic). We stress that the base field is not assumed to be 
algebraically closed (all the things considered here are collapsing to 
vacuous trivialities in the case of an algebraically closed base field).

Our notations are standard and largely follow Bourbaki \cite{bourbaki}. 
The symbols $\dotplus$, $\oplus$, and $\inplus$ denote 
direct sum of vector spaces, direct sum of Lie algebras, and 
semidirect sum of Lie algebras (the first summand acting on the second), 
respectively.

\section{Anisotropic algebras}

It is shown in \cite[Propostion 1.2]{farn-anis} that any anisotropic solvable Lie algebra
is abelian. From this and the Levi--Malcev decomposition follows that any anisotropic Lie 
algebra is reductive.

\begin{theorem}\label{th-reduct-anis}
For a reductive Lie algebra $L$ the following are equivalent:
\begin{enumerate}[\upshape (i)]
\item $L$ is anisotropic;
\item all proper subalgebras of $L$ are anisotropic;
\item all proper subalgebras of $L$ are reductive;
\item all $2$-dimensional subalgebras of $L$ are abelian;
\item $L$ does not contain subalgebras isomorphic to $\Sl(2)$.
\end{enumerate}
\end{theorem}

\begin{proof}
(i) $\Rightarrow$ (ii). If $S$ is a subalgebra of $L$, then for any $x\in S$,
$\ad_S x$ is a restriction of $\ad_L x$, hence the semisimplicity of the latter
implies the semisimplicity of the former.

(ii) $\Rightarrow$ (iii) follows from the observation above that any 
anisotropic Lie algebra is reductive.

(iii) $\Rightarrow$ (iv) follows from the obvious fact that a $2$-dimensional 
reductive Lie algebra is abelian.

(iv) $\Rightarrow$ (v) follows from the obvious fact that $\Sl(2)$
contains a $2$-dimensional nonabelian subalgebra.

(v) $\Rightarrow$ (i). Write $L$ as a direct sum $L = \mathfrak g \oplus A$,
where $\mathfrak g$ is semisimple and $A$ is abelian.
Suppose $\mathfrak g$ is not anisotropic. As $\mathfrak g$ contains semisimple 
and nilpotent components of each of its elements 
(\cite[Chapter I, \S 6, Theorem 3]{bourbaki}), $\mathfrak g$ contains a nonzero
nilpotent element, and by the Jacobson--Morozov theorem 
(\cite[Chapter VIII, \S 11, Proposition 2]{bourbaki}) $\mathfrak g$ contains 
$\Sl(2)$ as a subalgebra, a contradiction. Hence $\mathfrak g$ is anisotropic
and $L$ is anisotropic.
\end{proof}

Though the proof is elementary, and all the necessary ingredients are contained
in \cite{farn-anis} anyway (in particular, the implication 
(i) $\Rightarrow$ (iv) is noted in \cite[\S 1]{farn-anis}, and 
the equivalence (i) $\Leftrightarrow$ (v) in the case of semisimple $L$
is proved, with a slightly different argument, 
in \cite[Theorem 2.1]{farn-anis}), we find this explicit formulation of 
Theorem \ref{th-reduct-anis} interesting enough. There are many works in the 
literature devoted to study of \emph{minimal non-$\mathscr P$} Lie algebras,
i.e. Lie algebras not satisfying $\mathscr P$ and such that all their proper
subalgebras satisfy $\mathscr P$, where $\mathscr P$ is a certain ``natural''
property of Lie algebras (abelianity, nilpotency, solvability, simplicity, 
modularity of the lattice of subalgebras, \dots). In all the cases studied 
so far, the class of minimal non-$\mathscr P$ algebras turns out to be highly 
nontrivial (without further assumptions about the base field, such as 
algebraic or quadratic closedness, triviality of the Brauer group, etc.), with
lot of simple objects. To the contrary, from the Levi--Malcev decomposition and
Theorem \ref{th-reduct-anis} it follows that the class of minimal nonanisotropic
Lie algebras is relatively trivial: those are exactly solvable minimal 
nonabelian Lie algebras. One may ask a ``philosophical'' question: what makes 
the condition of being anisotropic different in that regard from other 
conditions? Where is a borderline for a property $\mathscr P$ which makes the 
class of minimal non-$\mathscr P$ Lie algebras small and ``simple'' (or even 
empty)?

\begin{corollary}
A simple Lie algebra all whose proper subalgebras are not simple, is either minimal
nonabelian, or isomorphic to $\Sl(2)$.
\end{corollary}

\begin{proof}
Let $L$ be a reductive Lie algebra all whose proper subalgebras are not simple.
By implication (v) $\Rightarrow$ (iii) of Theorem \ref{th-reduct-anis}, either $L$ is
isomorphic to $\Sl(2)$, or all proper subalgebras of $L$ are 
reductive. As any nonabelian reductive Lie algebra contains a simple subalgebra,
in the latter case all proper subalgebras of $L$ are abelian.
\end{proof}

In \cite[Theorem 2.2]{towers-agg} a statement similar to the corollary is proved about 
simple Lie algebras, all whose proper subalgebras are supersolvable.

\begin{theorem}\label{th-class-anis}
Let $\mathscr L$ be a nonempty class of Lie algebras satisfying the following properties:
\begin{enumerate}[\upshape (i)]
\item $\mathscr L$ is closed with respect to subalgebras;
\item 
if each proper subalgebra of a reductive Lie algebra $L$ belongs to $\mathscr L$,
then $L$ belongs to $\mathscr L$;
\item solvable Lie algebras belonging to $\mathscr L$ are abelian.
\end{enumerate}
Then $\mathscr L$ is the class of all anisotropic Lie algebras.
\end{theorem}

\begin{proof}
Any class of Lie algebras satisfying conditions (i) and (iii) consists of
anisotropic algebras. Indeed, from the Levi--Malcev decomposition and 
condition (iii) it follows that any algebra in $\mathscr L$ is reductive. 
Then from implication (iii) $\Rightarrow$ (i) of Theorem \ref{th-reduct-anis} 
and condition (i), it follows that any algebra in $\mathscr L$ is anisotropic.

Now, suppose that there is an anisotropic Lie algebra not belonging to 
$\mathscr L$, and consider such algebra $L$ of the minimal possible dimension. 
Then all proper subalgebras of $L$ belong to $\mathscr L$, and by condition (ii)
$L$ itself belongs to $\mathscr L$, a contradiction.
\end{proof}

\section{Regular algebras}\label{sec-regular}

If $N$ is a nilpotent subalgebra of a Lie algebra $L$, by $L^0(N)$ is denoted
the Fitting $0$-component with respect to the $N$-action on $L$ (i.e., the 
set of all elements of $L$ on which $N$ acts nilpotently).

Recall (\cite[Chapter VII, \S 2.2]{bourbaki}) that \emph{rank} $\rk L$ of a 
Lie algebra $L$ is the minimal possible non-vanishing power of the 
characteristic polynomial of $\ad x$, $x\in L$, and elements of $L$ for which
this minimal number is attained are called \emph{regular}. 
Another characterization of $x\in L$ to be a regular element is the equality 
$\dim L^0(x) = \rk L$.

If each nonzero element of $L$ is regular, then $L$ itself is called 
\emph{regular}.

It is clear that any nilpotent Lie algebra is regular, with rank equal the 
dimension of the algebra. If a regular Lie algebra $L$ is not
semisimple, i.e., contains a nonzero abelian ideal $I$, then for any $x\in I$,
$(\ad x)^2 = 0$, hence each element in $L$ is nilpotent, and by the Engel theorem
$L$ is nilpotent.
It is clear also that a regular semisimple Lie algebra is simple 
(see \cite[Chapter VII, \S 2.2, Proposition 7]{bourbaki}),
and that a regular simple Lie algebra is anisotropic 
(see \cite[Chapter VII, \S 2.4, Corollary 2]{bourbaki}). 

\begin{theorem}\label{th-regular}
For a simple Lie algebra $L$ the following are equivalent:
\begin{enumerate}[\upshape (i)]
\item $L$ is regular;
\item all proper subalgebras of $L$ are regular;
\item all proper subalgebras of $L$ are either simple, or abelian.
\end{enumerate}
\end{theorem}

\begin{proof}
(i) $\Rightarrow$ (ii) follows from the fact that if $S$ is a subalgebra of $L$,
and $x\in S$ is a regular element in $L$, then $x$ is a regular element in $S$ 
(\cite[Chapter VII, \S 2.2, Proposition 9]{bourbaki}).

(ii) $\Rightarrow$ (iii). By the observation above, any proper subalgebra of 
$L$ is either simple, or nilpotent. 
Hence $L$ does not contain a $2$-dimensional nonabelian Lie algebra,
and by implication (iv) $\Rightarrow$ (iii) of Theorem \ref{th-reduct-anis},
all proper subalgebras of $L$ are reductive, and all its nilpotent subalgebras 
are abelian.

(iii) $\Rightarrow$ (i). By implication (iii) $\Rightarrow$ (i) of 
Theorem \ref{th-reduct-anis}, $L$ is anisotropic. In any Lie algebra, Cartan 
subalgebras are exactly nilpotent subalgebras $N$ such that $L^0(N) = N$ 
(\cite[Chapter VII, \S 2.1, Proposition 4]{bourbaki}). But nilpotent subalgebras 
of $L$ are abelian, and $L^0(N)$ coincides with the centralizer of $N$, so 
Cartan subalgebras of $L$ are exactly abelian subalgebras coinciding with their
own centralizer. For an arbitrary nonzero element $x\in L$, its centralizer 
$\Z_L(x)$ cannot be simple, hence it is abelian. But, obviously, $\Z_L(x)$ 
coincides with its own centralizer, hence $\Z_L(x)$ is a Cartan subalgebra of 
$L$, $\dim \Z_L(x) = \dim L^0(x) = \rk L$, and $x$ is regular.
\end{proof}

Note that similar to the anisotropic case, minimal nonregular Lie algebras are
exactly solvable minimal nonnilpotent Lie algebras.

\begin{theorem}
Let $\mathscr L$ be a nonempty class of Lie algebras satisfying the following properties:
\begin{enumerate}[\upshape (i)]
\item $\mathscr L$ is closed with respect to subalgebras;
\item 
if each proper subalgebra of a Lie algebra $L$ belongs to $\mathscr L$, then
$L$ belongs to $\mathscr L$;
\item 
non-semisimple Lie algebras belonging to $\mathscr L$ are nilpotent.
\end{enumerate}
Then $\mathscr L$ is the class of all regular Lie algebras.
\end{theorem}

\begin{proof}
Any class of Lie algebras satisfying conditions (i) and (iii) consists of 
regular algebras. Indeed, from the Levi--Malcev decomposition and 
condition (iii) it follows that any algebra $L$ in $\mathscr L$ is either 
semisimple, or nilpotent. In the former case, write 
$L = \mathfrak g_1 \oplus \dots \oplus \mathfrak g_n$ as the direct sum of simple components. 
If $n>1$, by condition (i) the subalgebra of $L$ of the form $\mathfrak g_1 \oplus Kx$,
where $x$ is an arbitrary nonzero element of $\mathfrak g_2$, belongs to $\mathscr L$,
and by condition (iii) it is nilpotent, a contradiction. Hence $n=1$, that is,
$L$ is simple. By conditions (i) and (iii) $L$ does not contain $2$-dimensional
nonabelian subalgebra, and by implication 
(iv) $\Rightarrow$ (iii) of Theorem \ref{th-reduct-anis}, all subalgebras of $L$
are reductive. This, together with conditions (i) and (iii) again, implies that
all subalgebras of $L$ are either simple, or abelian, and by implication
(iii) $\Rightarrow$ (i) of Theorem \ref{th-regular}, $L$ is regular.

Now, the same elementary reasoning utilizing condition (ii) as at the end of 
the proof of Theorem \ref{th-class-anis}, shows that any regular Lie algebra 
belongs to $\mathscr L$.
\end{proof}

\section{Minimal nonabelian algebras}

It follows from the Levi--Malcev decomposition that any minimal nonabelian Lie 
algebra is either simple, or solvable. Solvable minimal nonabelian Lie algebras 
(even in a slightly more general minimal nonnilpotent setting) 
were described in \cite{stitzinger}, \cite{gkm}, and \cite{towers-laa}.
A simple minimal nonabelian Lie algebra is regular.
Simple minimal nonabelian Lie algebras were studied in \cite{farn-mna} and 
\cite{gein}, but their full description remains an open problem.

Recall that an algebra is called \emph{central} if its centroid coincides
with the base field. For simple algebras this is equivalent to the condition 
that the algebra remains simple under extension of the base field.

\begin{theorem}\label{th-b-d}
There are no central simple minimal nonabelian Lie algebras of types 
$\mathsf B_l$ ($l\ge 2$), $\mathsf C_l$ ($l\ge 3$, $l \ne 2^k$), 
$\mathsf D_l$ ($l\ge 5$, $l\ne 2^k$), $\mathsf G_2$, and $\mathsf F_4$.
\end{theorem}

\begin{proof}
The proof follows from the known classification of central simple Lie algebras
of these types (see, for example, \cite[Chapter IV]{seligman}).

\emph{Types $\mathsf B$--$\mathsf D$}. Each central simple Lie algebra
of this type (with the exception of $\mathsf D_4$) is isomorphic to a Lie
algebra of $J$-skew-symmetric elements $\sym^-(A,J) = \set{x\in A}{J(x) = -x}$,
where $A$ is a central simple associative algebra of dimension $n^2 > 16$ 
with involution $J$ of the first kind (smaller dimensions of $A$ are covered by 
``occasional'' isomorphisms between ``small'' algebras of different types, 
including type $\mathsf A$). By a known description of such algebras 
(see, for example, \cite[Theorem 5.1.23]{jacobson}), $A$ is isomorphic to 
$M_m(D)$, a matrix algebra of size $m \times m$ over a central division 
algebra $D$ with involution $j$, and $J$ has the form 
$$
(d_{k\ell})_{k,\ell=1}^m \mapsto \diag(g_1, \dots, g_m) (j(d_{k\ell}))^\top 
\diag(g_1^{-1}, \dots, g_m^{-1})
$$
for some $g_1, \dots, g_m \in D$ such that $j(g_k) = g_k$, $k=1,\dots,m$. 

If $D$ coincides with the base field, i.e. $A$ is a full matrix algebra, than
the Lie algebra $\sym^-(A,J)$ is split and, obviously, contains a lot of proper 
nonabelian subalgebras. Hence we may assume $\dim D \ge 4$.
From the description above it is clear that, provided $m>1$, the subalgebra $B$ 
of $A$ of all matrices with vanishing last row and column, is isomorphic to 
$M_{m-1}(D)$ and is stable under $J$, hence $\sym^-(B,J)$ is a Lie subalgebra 
of $\sym^-(A,J)$. Since $\dim A = m^2 \dim D \ge 25$, we have 
$\dim B = (m-1)^2 \dim D = s^2 \ge 9$, and this subalgebra is a 
central simple Lie algebra of dimension 
$\frac{s(s-1)}{2}$ or $\frac{s(s+1)}{2}$. Therefore, if $m>1$,
$\sym^-(A,J)$ contains proper nonabelian subalgebras, and it remains to consider
the case where $A = D$ is a division algebra.

Since $D$ has an involution, its exponent is equal to $2$, and its dimension 
$n^2$ is equal to some power of $4$. This excludes all the types mentioned in 
the statement of the theorem.

\emph{Type $\mathsf G_2$}.
Each central simple Lie algebra of this type is a derivation algebra of a 
$8$-dimensional Cayley algebra $\mathbb O$. The latter is obtained by the 
doubling (Cayley--Dickson) process from the $4$-dimensional associative quaternion algebra 
$\mathbb H$, and it is known that each derivation of $\mathbb H$ can be 
extended to a derivation of $\mathbb O$ 
(see, for example \cite[Theorem 2]{schafer}). Thus, $\Der(\mathbb O)$ always
contains a $3$-dimensional central simple Lie algebra $\Der(\mathbb H)$ as a 
subalgebra, and hence cannot be minimal nonabelian.

\emph{Type $\mathsf F_4$}.
Each central simple Lie algebra of this type is a derivation algebra of a 
$27$-dimensional exceptional simple Jordan algebra $\mathbb J$. It is known 
that derivations of $\mathbb J$ mapping a cubic subfield of $\mathbb J$
to zero form a central simple Lie algebra of type $\mathsf D_4$ 
(see, for example, \cite[Chapter IX, \S 11, Exercise 5]{jacobson-jordan}).
\end{proof}

We conjecture that the remaining types not covered by Theorem \ref{th-b-d} -- 
$\mathsf C_{2^k}$ and $\mathsf D_{2^k}$ -- cannot occur as well.

\begin{conjecture}
There are no central simple minimal nonabelian Lie algebras of types  
$\mathsf B$--$\mathsf D$ (except of $\mathsf D_4$).
\end{conjecture}

Let us provide some evidence in support of this conjecture.

\begin{lemma}
Let $D$ be a central division algebra of dimension $n^2$ over a field $K$ with 
involution $J$ of the first kind, such that $\sym^-(D,J)$ is a minimal 
nonabelian Lie algebra. Then for any $J$-symmetric or $J$-skew-symmetric element
$x$ in $D$, not lying in $K$, one of the following holds:
\begin{enumerate}[\upshape (i)]
\item $x$ is $J$-symmetric and of degree $2$;
\item $K[x]$ is of degree $\frac n2$, and $\dim_{K[x]} \C_D(x) = 4$;
\item $K[x]$ is a maximal subfield of $D$.
\end{enumerate}
\end{lemma}

\begin{proof}
The associative centralizer of $x$ in $D$, $\C_D(x)$, is a proper simple 
associative subalgebra of $D$. By the Double Centralizer Theorem 
(see, for example, \cite[\S 12.7]{pierce}), 
\begin{equation}\label{eq-double}
\dim K[x] \cdot \dim \C_D(x) = n^2 ,
\end{equation}
and the associative center of $\C_D(x)$ coincides with $K[x]$.

As $\C_D(x)$ is stable under $J$, $\sym^-(\C_D(x),J)$ is a Lie subalgebra of 
$\sym^-(D,J)$. If it coincides with the whole $\sym^-(D,J)$, then 
$\sym^-(D,J) \subseteq \C_D(x)$, and by (\ref{eq-double}), 
$\dim K[x] \le \frac{n^2}{\frac{n(n-1)}{2}} < 3$, hence $\dim K[x] = 2$, i.e.
$K[x]$ is a quadratic extension of $K$, the case (i). Note that in this case
$x$ cannot be $J$-skew-symmetric, as otherwise it lies in the Lie center of 
$\sym^-(D,J)$, a contradiction. 

If $\sym^-(\C_D(x),J)$ is a proper subalgebra of $\sym^-(D,J)$, then it is 
abelian, and by \cite[Theorem 2.2]{herstein}, $\C_D(x)$ is either commutative 
(i.e., a subfield of $D$), or is $4$-dimensional over its center $K[x]$. In the
former case, since the degree (= dimension) over $K$ of each intermediate field 
between $K$ and $D$ is $\le n$ (actually, a divisor of $n$), and since 
$K[x] \subseteq \C_D(x)$, we have $\dim K[x] = \dim \C_D(x) = n$, and 
$\C_D(x) = K[x]$, the case (iii). In the latter case, from (\ref{eq-double})
we have $\dim K[x] = \frac n2$ and $\dim \C_D(x) = 2n$, the case (ii).
\end{proof}

For example, if the division algebra $D$ is cyclic (what always happens 
over number fields), then, considering the conditions of the lemma 
simultaneously for a $J$-skew-symmetric element $x$ generating a cyclic 
extension of the base field, and even powers of $x$ (which are $J$-symmetric), 
one quickly arrives to a contradiction.

\medskip

For the remaining exceptional types, the question seems to be much more difficult, and it
is treated in \cite{garibaldi} using the language and technique of algebraic groups
and Galois cohomology.
There are central simple minimal nonabelian Lie algebras of types 
$\mathsf D_4$ and $\mathsf E_8$.
For types $\mathsf E_6$ and $\mathsf E_7$ partial answers are available.

Central simple minimal nonabelian Lie algebras of type $\mathsf A$ of the form 
$D^{(-)}/K 1$
(i.e., quotient of $D$, considered as a Lie algebra subject to commutator 
$[a,b]=ab-ba$, by the $1$-dimensional center spanned by the unit 1 of $D$), 
where $D$ is a central division associative algebra, were studied in 
\cite{gein}. A necessary, but not sufficient condition for such Lie algebra
to be minimal nonabelian is $D$ to be minimal noncommutative (i.e., all proper
subalgebras of $D$ are commutative). In this connection the following 
observation is of interest:

\begin{theorem}
Let $D$ be a central division associative algebra. Then the Lie algebra 
$D^{(-)}/K1$ is regular if and only if $D$ is a minimal noncommutative 
algebra.
\end{theorem}

\begin{proof}
Let the dimension of $D$ over the base field $K$ is equal to $n^2$, so 
$\dim D^{(-)}/K1 = n^2 - 1$. The Lie algebra 
$D^{(-)}/K1$ is regular if and only if the Lie centralizer of any nonzero 
element $\overline x\in D^{(-)}/K1$ is a Cartan subalgebra of 
dimension $n-1$, what, in associative terms, is equivalent to the condition 
that the associative centralizer $\C_D(x)$ of any element $x\in D\backslash K$, 
is a maximal subfield of $D$ of dimension $n$ over $K$. Taking this into 
account, the proof is an easy application of the Double Centralizer Theorem,
with reasonings similar to those used in the proof of the lemma above.

\emph{The ``only if'' part}. Suppose that for any $x\in D\backslash K$, 
$\C_D(x)$ is a maximal subfield of $D$. Consider a subfield 
$K[x] \subseteq \C_D(x)$ of $D$. We have $\C_D(x) = \C_D(K[x])$, and by the
Double Centralizer Theorem, $\dim K[x] \cdot \dim \C_D(x) = n^2$.
But the degree (= dimension) over $K$ of each intermediate field between $K$
and $D$ is $\le n$ (actually, a divisor of $n$), hence 
$\dim K[x] = \dim \C_D(x) = n$, and $\C_D(x) = K[x]$. That means that there are
no intermediate fields between $K$ and the maximal subfields of $D$.

If $A$ is a noncommutative proper subalgebra of $D$, then, obviously, $A$ is a 
division algebra. Its center $Z(A)$, being a field extension of $K$, either 
coincides with $K$, or is a maximal subfield of $D$. In the former case $A$
is central of dimension $m^2$, where $1 < m < n$, and its maximal subfield has 
degree $m$ over $K$, a contradiction. In the latter case, we have
$\dim A > \dim Z(A) = n$. Applying again the Double Centralizer Theorem, 
we have $\dim A \cdot \dim \C_D(A) = n^2$. Since
$Z(A) \subseteq \C_D(A)$, we have $\dim \C_D(A) \ge \dim Z(A) = n$, a contradiction.

\emph{The ``if'' part}. Suppose $D$ is minimal noncommutative. For an
arbitrary $x\in D$ not lying in the base field $K$, its centralizer $\C_D(x)$ is
a subfield of $D$. By the Double Centralizer Theorem, $\C_D(\C_D(x))$ is a 
simple subalgebra of $D$ (and, hence, is also a subfield), and 
$\dim \C_D(x) \cdot \dim \C_D(\C_D(x)) = n^2$. By the same argument as above
about degrees of intermediate fields between $K$ and $D$, 
$\dim \C_D(x) = \dim \C_D(\C_D(x)) = n$.
Since $\C_D(x) \subseteq \C_D(\C_D(x))$, this implies $\C_D(x) = \C_D(\C_D(x))$,
and $\C_D(x)$ is a maximal subfield of $D$.
\end{proof}

\section{Algebras of depth $2$}

Define the \textit{depth} of a Lie algebra in the following inductive way:
a Lie algebra has depth $0$ if and only if it is abelian, and has depth $n > 0$ 
if and only if it does not have depth $<n$ and all its proper subalgebras have depth
$<n$. Thus, minimal nonabelian Lie algebras are exactly algebras of depth $1$.

Many of the algebras considered below arise as semidirect sums $L \inplus V$ of a 
Lie algebra $L$ and an $L$-module module $V$ (in such a situation, we will 
always assume that $V$ is an abelian ideal: $[V,V] = 0$). It is clear that the 
depth of such semidirect sums is related to depth of $L$ and the maximal length
of chains of subspaces of $V$ invariant under action of subalgebras of $L$, 
though the exact formulation in the general case seems to be out of reach. 
In the particular case where $L$ is $1$-dimensional, the depth of such 
semidirect sum is equal to the maximal length of chains in $V$ of invariant 
subspaces with nontrivial $L$-action.

The following can be considered as an extension of the corresponding results
from \cite{stitzinger}, \cite{gkm}, and \cite{towers-laa}.

\begin{theorem}
A non-simple Lie algebra of depth $2$ over a field $K$ is isomorphic to one of 
the following algebras:
\begin{enumerate}[\upshape (i)]

\item\label{case-4dim-z0}
A $4$-dimensional solvable Lie algebra having the basis $\{x,y,z,t\}$ and the 
following multiplication table:
$$
[x,y] = z, \quad [x,z] = 0, \quad [y,z] = 0, \quad [z,t] = 0 ,
$$
with $\ad t$ acting on the space $Kx \dotplus Ky$ invariantly, without nonzero 
eigenvectors, and with trace zero.

\item\label{case-4dim-z1}
A $4$-dimensional solvable Lie algebra having the basis $\{x,y,z,t\}$ and the 
following multiplication table:
$$
[x,y] = z, \quad [x,z] = 0, \quad [y,z] = 0, \quad [z,t] = z ,
$$
with $\ad t$ acting on the space $Kx \dotplus Ky$ invariantly, without nonzero 
eigenvectors, and with trace $1$.

\item\label{case-sum-simple-mna-1}
A direct sum of a simple minimal nonabelian Lie algebra and $1$-dimensional
algebra.

\item\label{case-semidirect-nonabelian}
A semidirect sum $S \inplus V$, where $S$ is either the $2$-dimensional 
nonabelian Lie algebra, or a $3$-dimensional simple minimal nonabelian Lie 
algebra, and $V$ is an $S$-module such that each nonzero element of $S$ acts 
on $V$ irreducibly.

\item\label{case-semidirect-abelian}
A semidirect sum $S \inplus V$, where $S$ is an abelian $1$- or $2$-dimensional
Lie algebra, and $V$ is an $S$-module such that for each nonzero element 
$x\in S$, the maximal length of chains of $x$-invariant subspaces of $V$ is equal
to $2$ (what is equivalent to saying that any proper $x$-invariant subspace 
does not contain proper $x$-invariant subspaces).

\end{enumerate}
\end{theorem}

\begin{proof}
It is a straightforward verification that in each of these cases the 
corresponding Lie algebras have depth $2$, so let us prove that each non-simple 
Lie algebra $L$ of depth $2$ has one of the indicated forms.

Note that $L$ cannot be semisimple. For, in this case it is decomposed into the
direct sum of simple components: 
$L = \mathfrak g_1 \oplus \dots \oplus \mathfrak g_n$, $n>1$, 
and any subalgebra of the form $\mathfrak g_1 \oplus Kx$, 
$x\in \mathfrak g_2$, $x \ne 0$, is not minimal nonabelian.

Suppose that $L$ is non-semisimple and non-solvable, and let 
$L = \mathfrak g \inplus \Rad(L)$ be its Levi--Malcev decomposition. 
Then $\mathfrak g$ is minimal nonabelian and hence is simple. 
Further, $\Rad(L)$ abelian, as otherwise $\mathfrak g \inplus [\Rad(L),\Rad(L)]$
is a proper subalgebra of $L$ which is not minimal nonabelian.
Suppose now that $\rk \mathfrak g > 1$, and $\mathfrak g$ acts on $\Rad(L)$ nontrivially.
Then taking $x\in \mathfrak g$ with a nontrivial action on $\Rad(L)$,
and the Cartan subalgebra $H$ of $\mathfrak g$ of dimension $>1$ containing $x$,
we get a subalgebra $H \inplus \Rad(L)$ of $L$ which is not minimal nonabelian.
Hence in the case $\rk \mathfrak g > 1$, $\Rad(L)$ is a trivial 
(and then, obviously, $1$-dimensional) $\mathfrak g$-module, and we arrive at
case (\ref{case-sum-simple-mna-1}). 
If $\rk \mathfrak g = 1$, then $\mathfrak g$ is $3$-dimensional.
If some nonzero $x\in \mathfrak g$ acts on $\Rad(L)$ trivially, then so is
$[x,\mathfrak g]$, and, since $\mathfrak g$ is 
generated by the latter subspace, the whole $\mathfrak g$ acts on $\Rad(L)$ 
trivially, a case covered by (\ref{case-sum-simple-mna-1}). Assume that
any nonzero $x\in \mathfrak g$ acts on $\Rad(L)$ nontrivially. 
The Lie subalgebra $Kx \inplus \Rad(L)$ contains, in its turn, a subalgebra 
$Kx \inplus V$ for any proper $\ad x$-invariant subspace $V$ of $\Rad (L)$, 
what shows that $x$ acts trivially on $V$. Letting here $V$ to be 
the Fitting $1$-component with respect to the $x$-action on $\Rad(L)$,
we see that $\Rad(L) = V$, what means that $x$ acts on 
$\Rad(L)$ nondegenerately, and hence, irreducibly.
We arrive at case (\ref{case-semidirect-nonabelian}).

It remains to consider the case of $L$ solvable. Take any subspace $A$ of 
$L$ of codimension $1$ containing $[L,L]$, and a complimentary $1$-dimensional
subspace: 
\begin{equation}\label{eq-sum}
L = Kt \dotplus A ,
\end{equation} 
$\ad t$ acts on $A$. Since $A$ is a proper ideal 
of $L$, it is either abelian or minimal nonabelian. In the former case, we 
arrive at the semidirect sum $Kt \inplus A$, and it is
easy to see that any proper nonabelian subalgebra of $L$ is isomorphic to the 
semidirect sum $Kt \inplus V$, where $V$ is a proper 
$\ad t$-invariant subspace of $A$. Thus, for $L$ to be of depth $2$ is 
equivalent to the condition described in case (\ref{case-semidirect-abelian})
(with $S$ $1$-dimensional).

Suppose now that $A$ is minimal nonabelian. According to \cite[Theorem 4]{gkm} 
(also implicit in \cite{stitzinger} and \cite{towers-laa}), each solvable 
minimal nonabelian Lie algebra is either isomorphic to the $3$-dimensional 
nilpotent Lie algebra, or to the semidirect sum $Kx \inplus V$ such that 
$\ad x$ acts on $V$ irreducibly (in particular, $\ad x|_V$ is nondegenerate). 
Further, $\ad t$ is a derivation of $A$, 
and subtracting from $t$ an appropriate element of $A$, we may assume that 
either $t$ is central, i.e. (\ref{eq-sum}) is the direct sum of $A$ and 
$1$-dimensional algebra, or $\ad t$ is an outer derivation of $A$.

Suppose first that $A$ is $3$-dimensional nilpotent, i.e., has a basis 
$\{x,y,z\}$ with multiplication table $[x,y] = z, [x,z] = [y,z] = 0$.
If $t$ is central, we arrive at a particular case of (\ref{case-4dim-z0}). 
Straightforward computation shows that each outer derivation of $A$ is 
equivalent to a derivation $d$ which acts invariantly on the space 
$Kx \dotplus Ky$, and either $d|_{Kx \dotplus Ky}$ has trace zero, and 
$d(z) = 0$, or $d|_{Kx \dotplus Ky}$ has trace $1$, and $d(z) = z$. 
These two cases correspond to the cases 
(\ref{case-4dim-z0}) and (\ref{case-4dim-z1}) respectively, 
with the condition of absence of nonzero eigenvectors to ensure 
the absence of subalgebras which are not minimal nonabelian.

Suppose now that $A = Kx \inplus V$, $\ad x$ acts on $V$ 
irreducibly. If $t$ is central, $L \simeq Kx \inplus (V \dotplus Kt)$ 
(with $\ad x$ acting on $t$ trivially), a case covered by 
(\ref{case-semidirect-abelian}) (with $S$ $1$-dimensional). 
Straightforward computation shows that each outer derivation of $A$ is 
equivalent to a derivation $d$ which acts on $V$ invariantly, and either 
$[\ad x,d] = 0$ in the Lie algebra $\gl(V)$, and $d(x) = 0$, or 
$[\ad x,d] = \ad x$ and $d(x) = x$. These two cases correspond to the cases 
(\ref{case-semidirect-abelian}) and (\ref{case-semidirect-nonabelian})
respectively (with $S$ $2$-dimensional), with the respective conditions to 
ensure the absence of subalgebras which are not minimal nonabelian.
\end{proof}

\begin{corollary}[to Theorems \ref{th-reduct-anis} and \ref{th-regular}]
A simple Lie algebra of depth $2$ is either isomorphic to $\Sl(2)$, or regular.
\end{corollary}

\begin{proof}
It is clear that $\Sl(2)$ has depth $2$. Hence a simple Lie algebra
$L$ of depth $2$ is either isomorphic to $\Sl(2)$, or does not contain 
$\Sl(2)$ as a proper subalgebra. In the latter case, by 
implication (v) $\Rightarrow$ (iii) of Theorem \ref{th-reduct-anis}, 
all subalgebras of $L$ are reductive. But as each minimal nonabelian Lie
algebra is either simple, or solvable, all subalgebras of $L$ are either
simple, or abelian, and by 
implication (iii) $\Rightarrow$ (i) of Theorem \ref{th-regular}, $L$ is regular.
\end{proof}

In group theory, a notion analogous to depth in the class of finite $p$-groups 
is called \emph{$\mathcal{A}_n$-groups}, see \cite[\S 65]{bj} for their 
discussion and for a partial description of $\mathcal{A}_2$-groups.

\section*{Acknowledgements}

This work was essentially done more than 20 years ago, while being a student at
the Kazakh State University under the guidance of Askar Dzhumadil'daev, and has
been reported at a few conferences in the Soviet Union at the end of 1980s. 
However, I find the results interesting and original enough even today to be 
put in writing, with some minor additions and modifications. 
During the final write-up I was supported by grants ERMOS7 
(Estonian Science Foundation and Marie Curie Actions) and ETF9038 
(Estonian Science Foundation).

\end{document}